\documentclass{amsart}
\usepackage{mathrsfs}
\usepackage{amssymb}
\usepackage{hyperref}
\usepackage{url, array}
\usepackage[T1]{fontenc}
\usepackage{color}
\usepackage{comment}
\usepackage[all]{xy}
\usepackage{mathtools}
\newtheorem{thm}{Theorem}[section]
\newtheorem{lemma}[thm]{Lemma}
\newtheorem{cor}[thm]{Corollary}
\newtheorem{prop}[thm]{Proposition}
\DeclareMathOperator{\id}{id}

\theoremstyle{definition}
\newtheorem*{summary*}{Summary}
\newtheorem{defi}[thm]{Definition}
\newtheorem{example}[thm]{Example}

\theoremstyle{remark}
\newtheorem*{remark*}{Remark}

\DeclareMathOperator*{\colim}{colim}
\DeclareMathOperator{\Ord}{Ord}
\DeclareMathOperator{\Hom}{Hom}
\DeclareMathOperator{\Aut}{Aut}
\DeclareMathOperator{\Iso}{Iso}

\newcommand{\fieldk}{\mathbb{K}}
\newcommand{\ul}{\underline}
\newcommand{\bcirc}{\:\bar{\circ}\:}

\newcommand{\spe}{\mathsf}

\newcommand{\assspecies}{\spe{Ass}}
\newcommand{\commspecies}{\spe{Com}}
\newcommand{\Ii}{\spe{I}}
\newcommand{\Nn}{\spe{N}}
\newcommand{\psp}{\spe{W}}
\newcommand{\Xx}{\spe{X}}
\newcommand{\speciesa}{\spe{F}}
\newcommand{\speciesb}{\spe{G}}
\newcommand{\speciesc}{\spe{H}}
\newcommand{\operadspecies}{\spe{P}}
\newcommand{\cooperadspecies}{\spe{C}}
\newcommand{\coopspecies}{\cooperadspecies}

\newcommand{\ope}{\mathcal}

\newcommand{\counitcooperad}{\ope{I}}
\newcommand{\com}{\ope{C}om}
\newcommand{\comm}{\com}
\newcommand{\cocomm}{co\ope{C}om}
\newcommand{\ass}{\ope{A}ss}
\newcommand{\coass}{co\ope{A}ss}
\newcommand{\operad}{\ope{P}}
\newcommand{\coop}{\ope{C}}
\newcommand{\cooperad}{\coop}
\newcommand{\Cc}{\coop}
\newcommand{\wordcooperad}{\ope{W}}
\newcommand{\noncrossingcooperad}{\ope{N}}

\newcommand{\coalgebras}{\text{-coalgebras}}
\newcommand{\coalga}{X}
\newcommand{\groundthing}{\mathbb{A}}

\newcommand{\category}{\mathbf{Lin}}

\title{A non-crossing word cooperad for free homotopy probability theory}
\author{Gabriel C. Drummond-Cole}
\thanks{This work was supported by IBS-R003-D1.}
\address{Center for Geometry and Physics, Institute for Basic Science (IBS), Pohang 37673, Republic of Korea}
\begin{document}
\begin{abstract}
We construct a cooperad which extends the framework of homotopy probability theory to free probability theory. The cooperad constructed, which seems related to the sequence and cactus operads, may be of independent interest.
\end{abstract}
\maketitle
\section*{Introduction}
The purpose of this paper is to provide a convenient operadic framework for the cumulants of free probability theory. In~\cite{DrummondColeParkTerilla:HPTI,DrummondColeParkTerilla:HPTII}, the author and his collaborators described an operadic framework for classical and Boolean cumulants. This framework involves a choice of governing cooperad, and in both the classical and Boolean cases, the choice is an ``obvious'' and well-studied algebraic object. Namely, for classical cumulants, the governing cooperad is the cocommutative cooperad, while for Boolean cumulants it is the coassociative cooperad.

Extending this framework to free probability requires the construction of a governing cooperad with certain properties. The main construction of this paper is a cooperad, called the \emph{non-crossing cooperad}, satisfying these properties. As far as the author can tell, this cooperad is, at least to some degree, new. No well-studied cooperad (such as those in~\cite{Zinbiel:ETA2010}) seems to satisify the requisite properties.  That said, there is clearly some sort of relationship between the newly constructed cooperad and the sequence~\cite{McClureSmith:MCOLNC,BergerFresse:COAC} and cactus~\cite{Voronov:NUA,Kaufmann:OSVCTR,Kaufmann:OSCDCCKHA} operads. This line of thinking is not pursued in this article beyond the remark at the end of Section~\ref{section:words}. If it turns out that this cactus variant is well-known, that would be delightful---please let us know. 

Also, we make no attempt here to axiomatize the properties necessary to interface appropriately with free probability or to prove any uniqueness results. That is to say, there is every likelihood that this is the ``wrong'' cooperad. First of all, there is the near miss in terms of structure compared to the previously known operads. In addition, there are at least two failures of parallelism between the classical and Boolean cases and the new case presented here. See the remark following Theorem~\ref{thm: freever}. One possible explanation for these failures is that the correct framework requires \emph{operator-valued} free cumulants, that is, free cumulants with a not necessarily commutative ground ring. This line of reasoning will be pursued in other work~\cite{DrummondCole:OAOVFC}. It would also be exciting to hear about other potential frameworks to bring free cumulants into the framework of this kind of operadic algebra, whether along the same rough lines as in this paper or not.

The remainder of the paper is organized as follows. We describe the kind of words we will use and construct two cooperads spanned by them. The first, the \emph{word cooperad}, is auxilliary for our purposes although it may have independent interest. We construct the \emph{non-crossing word cooperad} as a quotient of the word cooperad. After a brief review of necessary notions from homotopy probability theory and free probability theory, we apply the non-crossing word cooperad to the motivating question and show that it fits into the framework of homotopy probability theory. 

For convenience, we work with \emph{unbiased} definitions of operads and cooperads, writing them in terms of finite sets and never choosing a particular ordered set. This is not usual in the literature although it should be familiar to experts. We conclude with an appendix describing only those aspects of this theory necessary for the current paper.
\subsection*{Conventions}
We will use the notation $[n]$ to denote the set $\{1,\ldots, n\}$. We work over a field $\fieldk$ of characteristic zero.

\subsection*{Acknowledgements}
The author gratefully acknowledges useful conversations with Joey Hirsh, John Terilla, Jae-Suk Park, and Ben Ward.

\section{Words and their cooperads}\label{section:words}

\subsection{Words}
This section establishes some basic definitions and lemmas about words.

A word $w$ is a nonempty finite sequence of elements from a set $S$. In this context, $S$ is called the \emph{alphabet} and elements of $S$ or the sequence $w$ are called letters.

The word $w$ is {\em pangrammatic} if it contains each letter from the alphabet $S$.
\begin{defi}
The word $w$ is {\em reduced} if it has no subword of the form $aa$ and either is length one or has different first and last letters. 

The {\em reduction} $\overline{w}$ of the word $w$ is the unique minimal length word obtained by repeated reduction by 
\begin{eqnarray*}
\ldots aa\ldots &\mapsto &\ldots a\ldots 
\\a\ldots a&\mapsto& a\ldots
\end{eqnarray*}
In the second case, $a$ must be the first and last letter of $w$; this relation is not a ``local'' move on subwords.
\end{defi}
\begin{defi}
The word $w$ is {\em non-crossing} if it never contains
\[\ldots a \ldots b\ldots a\ldots b\ldots\]
for distinct $a$ and $b$ in $S$. 

A word is {\em crossing} unless it is non-crossing.
\end{defi}

\begin{remark*}
A map of sets $f:S\to T$ induces a map from words in $S$ to words in $T$, which will be also denoted by $f$.
\end{remark*}

\begin{defi}
Let $w$ be a word on the alphabet $T$ and let $S$ be a subset of the alphabet $T$ which contains at least one letter of $w$. Then $w|S$, called {\em the word restricted to $S$}, is the word obtained by deleting all letters not in $S$.
\end{defi}
If a word $w$ is pangrammatic then $w$ can be restricted to any nonempty subset of the alphabet and the result is pangrammatic.

The following lemmas about reduction, restriction, and words induced by functions, are immediate.
\begin{lemma}\label{lemma:reduce function reduce}
Let $w$ be a word on the alphabet $S$ and let $f$ be a map of sets $S\to T$. Then $\overline{f(\overline{w})}=\overline{f(w)}$.
\end{lemma}
\begin{lemma}\label{lemma:reduce restrict reduce}
Let $w$ be a word on the alphabet $T$ and let $S$ be a subset of $T$ containing at least one letter from $w$. Then $\overline{\overline{w}|S}=\overline{w|S}$.
\end{lemma}
\begin{lemma}\label{lemma:map restrict commute}
Let $w$ be a word on the alphabet $R$, let $f$ be a map of sets $R\to S$, and let $T$ be a subset of $S$ containing at least one letter of $f(R)$. Then $f(w|f^{-1}(T))=f(w)|T.$
\end{lemma}
In general, we do not have $f(w|S)=f(w)|f(S)$ unless $S=f^{-1}f(S)$.

\subsection{The word cooperad}
Now we construct a cooperad spanned by a class of words. In Section~\ref{subsec noncrossing word cooperad}, we construct a second, closely related cooperad which will be our main point of interest. We work over the field $\fieldk$ with an unbiased definition for cooperads. See Appendix~\ref{appendix:coop} for details. As stated in the introduction, there is some relationship between the cooperads constructed here and the sequence and cactus operads. As the relationship is not entirely clear, the following is a self-contained presentation. There is a remark about the connection at the end of Section~\ref{section:words}.

\begin{defi}
The \emph{word species} is the species $\psp$ constructed as follows. To a finite set $S$, the functor $\psp$ assigns the $\fieldk$-vector space spanned by pangrammatic reduced words on $S$. We define the structure necessary to make this species a cooperad, the \emph{word cooperad} $\wordcooperad$, showing coassociativity in Proposition~\ref{prop: reduced words cooperad} below.

The decomposition map $\psp\to\psp\bcirc \psp$ can be specified, as discussed in the appendix, by defining $\Delta_f$ for each surjection $f:S\twoheadrightarrow T$. We define $\Delta_f$ as follows.
\[
\Delta_{f}(w) = 
\overline{f(w)}\otimes
\bigotimes_{t\in T} \overline{w|f^{-1}(t)}.
\]

The \emph{counit map} $\epsilon$, for $|S|=1$, takes the unique word in $\psp(S)$ to $1\in I(S)$.
\end{defi}
Checking equivariance with respect to both isomorphisms $S\to S'$ and isomorphisms $T\to T'$ under $S$ is straightforward, so the decomposition map $\Delta$ is well-defined.
\begin{example}
Let $S=\{a_1,a_2,a_3\}$ and let $w= a_1a_2a_1a_3$. Then the limit of interest can be specified in terms of five choices of $T$ and a surjection. $S\to T$. These are:
\begin{itemize}
\item the constant map $f_0:S\to \{b_0\}$,
\item the three maps $f_{ij}:S\to T_{ij}=\{b_{ij},b_k\}$ which take $a_i$ and $a_j$ to $b_{ij}$ and $a_k$ to $b_k$, and
\item the map $f_3=S\to T_3=\{b_1,b_2,b_3\}$ which takes $a_i$ to $b_i$.
\end{itemize}
Then $\Delta w$ is (represented by) the sum of $\Delta_{f_*}$ over these five choices of $f_*$. That is:

\begin{align*}
\Delta w &=& b_0 \otimes& \underbrace{w}_{b_0} 
\\&+& b_{12}b_3\otimes& \left(\underbrace{a_1a_2}_{b_{12}}\otimes \underbrace{a_3}_{b_3}\right)
\\&+& b_{13}b_2\otimes&\left(\underbrace{a_1a_3}_{b_{13}}\otimes \underbrace{a_2}_{b_2}\right)
\\&+& b_{1}b_{23}b_1b_{23}\otimes& \left(\underbrace{a_1}_{b_1}\otimes\underbrace{a_2a_3}_{b_{23}}\right)
\\&+& b_1b_2b_3\otimes& \left(\underbrace{a_1}_{b_1} \otimes \underbrace{a_2}_{b_2}\otimes \underbrace{a_3}_{b_3}\right).
\end{align*}

\end{example}
\begin{prop}\label{prop: reduced words cooperad}
The decomposition map and the counit map give $\wordcooperad=(\psp,\epsilon,\Delta)$ the structure of a cooperad.
\end{prop}
\begin{proof}
It suffices to show coassociativity holds separately on each individual factor in the limit making up $\psp\bcirc\psp\bcirc\psp$. Given a word $w$ in $S$ and surjections $S\stackrel{f}{\twoheadrightarrow} T\stackrel{g}{\twoheadrightarrow} U$, we have the following two compositions of decompositions:
\[
\left(\Delta_g\otimes \id\right)
\Delta_{f} (w)=
\overline{g(\overline{f(w)})}\otimes
\bigotimes_{u\in U}\overline{\overline{f(w)}|g^{-1}(u)}
\otimes
\bigotimes_{t\in T} \overline{w|f^{-1}(t)}
\]
and
\begin{multline*}
\left(\id\otimes\bigotimes_{u\in U}\Delta_{f|_{(gf)^{-1}(u)}}\right)
\Delta_{gf}(w)
\\=
\overline{gf(w)}\otimes
\bigotimes_{u\in U} \left(\overline{f(\overline{w|(gf)^{-1}(u)})}
\otimes\bigotimes_{t\in g^{-1}(u)} \overline{\overline{w|(gf)^{-1}(u)}|f^{-1}(t)}\right)
\\=
\overline{gf(w)}\otimes
\bigotimes_{u\in U} \overline{f(\overline{w|(gf)^{-1}(u)})}
\otimes\bigotimes_{t\in T} \overline{\overline{w|(gf)^{-1}(g(t))}|f^{-1}(t)}.
\end{multline*}
To show coassociativity, we will show that the terms in the product match up individually. This means that there are three easy verifications to make. First, it is a direct application of Lemma~\ref{lemma:reduce function reduce} that \[\overline{g(\overline{f(w)})}=\overline{gf(w)}.\]

Second, using Lemmas~\ref{lemma:reduce function reduce},~\ref{lemma:reduce restrict reduce}, and~\ref{lemma:map restrict commute}, we see 
\[
\overline{\overline{f(w)}|g^{-1}(u)}
=
\overline{f(w)|g^{-1}(u)}
=
\overline{f(w|(gf)^{-1}(u))}
=
\overline{f(\overline{w|(gf)^{-1}(u)})}.
\] 
Finally, using Lemma~\ref{lemma:reduce restrict reduce} again, we see that 
\[\overline{\overline{w|(gf)^{-1}(g(t))}|f^{-1}(t)}
=
\overline{\left(w|(gf)^{-1}(g(t))\right)|f^{-1}(t)}
=
\overline{w|f^{-1}(t)}.
\]
We omit the verification of counitality. 
\end{proof}
 \subsection{The non-crossing word cooperad}\label{subsec noncrossing word cooperad}
\begin{defi}
The {\em non-crossing species} $\Nn$ assigns to the set $S$ the $\fieldk$-vector space spanned by pangrammatic reduced non-crossing words on $S$. Similarly, the {\em crossing species} $\Xx$ assigns to $S$ the span of pangrammatic reduced crossing words on $S$.
\end{defi}
There is a natural inclusion of $\Xx$ into $\psp$ whose quotient is isomorphic to $\Nn$.
\begin{prop}\label{prop:noncrossing cooperad}
The quotient map $\psp\to \Nn$ makes the non-crossing species a quotient cooperad of the word cooperad.
\end{prop}
\begin{proof}
$\Xx(1)$ is zero dimensional so the counit descends to the quotient. 

Let $w$ be an arbitrary crossing word in the alphabet $S$. Then it is only necessary to show that $\Delta(w)$ is in the kernel of the map $\psp\bcirc \psp\to \Nn\bcirc\Nn$. 
The word $w$ contains the pattern $\ldots a\ldots b\ldots a\ldots b\ldots$ for distinct $a$ and $b$ in $S$. Consider $\Delta_{f}(w)$ for some surjection $f:S\to T$. If $f(a)\ne f(b)$ then $f(w)$ and hence its reduction $\overline{f(w)}$ is crossing. On the other hand, if $f(a)=f(b)$ then $f|f^{-1}f(a)$ and hence its reduction $\overline{f|f^{-1}f(a)}$ is crossing. Therefore $\Delta(w)$ is contained in $\Xx\bcirc \psp +\psp\bcirc \Xx$.
\end{proof}
\begin{defi}
We call $\noncrossingcooperad=(\Nn,\epsilon,\Delta)$, where $\epsilon$ and $\Delta$ are induced by the quotient map $\psp\to \Nn$, the \emph{non-crossing word cooperad}.
\end{defi}
The following is a direct calculation.
\begin{lemma}
Let $w$ be a pangrammatic non-crossing word on the alphabet $T$ and let $S$ be a subset of $T$. Then $w|S$ is non-crossing.
\end{lemma}
\begin{cor}\label{cor:characterization of nc Delta}
The decomposition map of the non-crossing word cooperad applied to the word $w$ is the limit of $\Delta^\mathrm{nc}_{f}(w)$, where $\Delta^\mathrm{nc}_f(w)$ is equal to $\Delta_f(w)$ if $f(w)$ is non-crossing and $0$ if $f(w)$ is crossing. 
\end{cor}
\begin{remark*}
Both of the operads constructed here clearly have some relationship to the sequence operad~\cite{McClureSmith:MCOLNC} and cactus operad~\cite{Voronov:NUA,Kaufmann:OSVCTR}. This is perhaps easiest to see with the very clean presentation in~\cite{GalvezCarrilloLombardiTonks:AIOSC}. There the authors describe two operads whose underlying species differ from those considered here only by allowing words to begin and end with the same letter. 

From either a cactus or sequence perspective, the subspecies specified by this additional condition forms a suboperad. For surjections, which are described combinatorially, the condition itself probably gives the best description. For cacti, one can say that it is the suboperad of cellular chains of spineless cacti where the global root coincides with some intersection of lobes. 

Based on this, a naive guess might be that the cooperad here is the dual of the appropriate suboperad. However, the decomposition is \emph{not} dual to the composition map of sequences or cacti and indeed does not even respect the grading of the operations or cells. So the relationship, should it exist, must be subtler than that. Ben Ward has pointed out that the suboperad of ``generic'' cacti, where no more than two cactus lobes can meet at a point, is dual to an appropriately defined subcooperad of the non-crossing word cooperad.
\end{remark*}
\section{Review of (homotopy) probability theory}
This section consists of the glue directly connecting what we have set up to our main application. First we review an operadic framework for homotopy probability theory, and then recall the {free cumulants}, which govern free independence in non-commutative probability theory. 
\subsection{Review of homotopy probability theory}
We recall in a few words the setup of homotopy probability theory in operadic terms. Again, see Appendix~\ref{appendix:coop} for details about the conventions for operads and cooperads. 

Homotopy probability theory was introduced by Park~\cite{Park:L} as a simplification of his algebraic model for quantum field theory where Planck's constant plays no role. The most complete reference is Park's monograph~\cite{Park:HTPSICIHLA}, which differs in both notation and definitions from this paper but agrees in spirit with what is here.

One of Park's motivations was to generalize and properly axiomatize (algebraic) probability spaces in terms of homotopy algebra. The following is a ``classical'' definition before generalization (see, for example,~\cite{NicaSpeicher:LCFP}).
\begin{defi}
A \emph{non-commutative probability space} (respectively, a \emph{commutative algebraic probability space}) is a unital associative (unital commutative associative) $\fieldk$-algebra $V$ equipped with a unit-preserving linear map $E$ from $V$ to $\fieldk$. We assume no further compatibility between the linear map and the algebra structure. The elements of $V$ are called \emph{random variables} and the map $E$ is called the \emph{expectation}.
\end{defi}
\begin{remark*}Since commutative algebraic probability spaces most typically arise as measurable functions on a measure space they are often defined to satisfy additional analytic properties that we will ignore here. See e.g.,~\cite{Tao:APS}.
\end{remark*}
Two basic ingredients of the motivation to generalize this definition come from physics, where the random variables are the \emph{observables} in a quantum field theory. 

First of all, usually a field theory possesses physical symmetries. For symmetries of the classical action, this is an old and well-known part of the BV-BRST formalism that can be dealt with by introducing so-called ghosts. This amounts to replacing the linear space of observables with a chain complex. 

There is another kind of symmetry that may come into play, namely symmetry of the expectation. In particular, we only expect closed elements in the complex to be observables, and we expect boundaries in the chain complex to be trivial observables (in well-behaved cases, the converse should also be true, at least morally). This symmetry of the expectation is probably less understood and analyzed in these terms than symmetry of the action. See~\cite[Section 6]{Park:HTPSICIHLA} for some discussion of this point. 

In the following definition, a unital version of a definiton in~\cite{DrummondColeParkTerilla:HPTI}, we stick to the associative framework, but there is clearly a commutative variation.

\begin{defi}
A \emph{unital associative homotopy probability space} is a unital graded associative $\fieldk$-algebra equipped with a differential which kills the unit and a unit-preserving chain map to the ground field.
\end{defi}
A unital associative homotopy probability space concentrated in degree zero is precisely a non-commutative probability space as defined above.

However, this definition cannot capture the full subtlety of the observables in a quantum field theory. Usually, the symmetries of the action are not be compatible with the product, so that the product of observables may not be observables (the product of closed elements may not be closed). Instead, the product may need to be ``corrected'' in some way to be fully defined. Homotopy probability theory can be traced back to Park's observation of this problem and a potential solution for it in~\cite{Park:FFQFTQDA}. 

 One way to deal with the problem of correcting the classical product is via homotopy algebra, which gathers together these corrections into a coherent package. But this leads naturally to an algebraic generalization where there is not a single product out of which many products can be built, but rather a binary product, an independent trilinear product, and so on. Again, this point of view is espoused at much greater length and in more detail in~\cite{Park:HTPSICIHLA}. Following Park, here we take a broad view and treat this system of corrections as a black box, defining the algebraic structure as minimally as possible.

The following definition defines our spaces of random variables or observables along with mock products, which basically don't need to satisfy any algebraic identities or respect the differential. See~\ref{subsec:auto} for the definition of strong coaugmentation and the notation below. 
\begin{defi}\label{defi:CCor}
Let $\Cc=(\coopspecies,\epsilon,\Delta)$ be a strongly coaugmented cooperad. A \emph{$\Cc$-correlation algebra} is a chain complex $V$ equipped with a degree zero linear map (not necessarily a chain map) $\varphi_V:{\coopspecies}\circ V\to V$ such that, for $|S|=1$, we have 
\[
 V\cong \coopspecies_S\circ V\to\coopspecies\circ V\xrightarrow{\varphi_V}V
 \] is the identity.
\end{defi}
Next, we encode the expectation.
\begin{defi}\label{defi:CProb}
Let $\Cc=(\coopspecies, \epsilon,\Delta)$ be a cooperad. Fix a $\Cc$-correlation algebra $\groundthing$. An \emph{$\groundthing$-valued homotopy $\Cc$-probability space} is a $\Cc$-correlation algebra $(V,\varphi_V)$ equipped with
\begin{enumerate}
\item a map $\eta$ of chain complexes $\groundthing\to V$, called the \emph{unit}, such that $\varphi_V\circ \Cc\eta=\eta\circ \varphi_\groundthing$ and
\item a map $E$ of chain complexes from $V$ to $\groundthing$, called the \emph{expectation}, such that $E\circ \eta=\id_\groundthing$. 
\end{enumerate}
The conditions on the maps $\eta$ and $E$ are equivalent to the commutativity of the following diagram:
\[\xymatrix{
\coopspecies\circ\groundthing\ar[d]_{\coopspecies\circ\eta}\ar[r]^{\varphi_\groundthing}
&\groundthing \ar[d]_\eta\ar[r]^{\id_\groundthing}
&\groundthing
\\
\coopspecies\circ V\ar[r]_{\varphi_V}& V\ar[ur]_E	
}
\]
\begin{remark*}Definitions~\ref{defi:CCor}~and~\ref{defi:CProb} provide definitions for homotopy probability theory over an arbitrary cooperad. The case of the coassociative cooperad was addressed in~\cite{DrummondColeParkTerilla:HPTI}; the case of the cocommutative cooperad was addressed in \cite{DrummondColeParkTerilla:HPTII,DrummondColeTerilla:CIHPT}. The specialization of the definition given here to the appropriate cooperads is \emph{not} equivalent to the definitions given there. Rather, the definition here is more general. See Remark 2 of~\cite{DrummondColeParkTerilla:HPTII}. Park~\cite{Park:HTPSICIHLA} addresses the case of the cocommutative cooperad at a roughly comparable level of generality. 

This article is only intended to establish a relationship between the noncrossing word cooperad and free cumulants. It is not intended to establish a full homotopy probability theory in the free setting. Because of this, the recollection below may be too terse for some. Therefore, regardless of any differences in definitions, the interested or puzzled reader is advised to consult the references above (especially the monograph~\cite{Park:HTPSICIHLA}) for more details about homotopy probability theory.
\end{remark*}

Now let $V$ be an $\groundthing$-valued homotopy $\Cc$-probability space. The \emph{cumulant morphism} is the $\Cc$-coalgebra map $\tilde{K}$ (or its adjoint $K:\coopspecies\circ V\to \groundthing$) that fits into the following diagram of $\Cc$-coalgebras (well-defined because $\tilde{\varphi}_{\groundthing}$ is an automorphism by Lemma~\ref{lemma: identify coalgebra iso}):
\begin{equation}\label{eq: definecumulants}\begin{gathered}
\xymatrix{
\Cc\circ \groundthing\ar[r]^{\tilde{\varphi}_{\groundthing}}
&
\Cc\circ\groundthing
\\
\Cc\circ V\ar@{.>}[u]^{\tilde{K}}\ar[r]_{\tilde{\varphi}_V}
&
\Cc\circ V\ar[u]_{\tilde{E}=\Cc\circ E}.
}
\end{gathered}\end{equation}
\end{defi}
\begin{example}\label{ex:BoolCum} 
\begin{enumerate}
\item We reinterpret a unital associative homotopy probability space $(V,\eta,E)$ in our current framework. Since the underlying species of $\ass$ and $\coass$ are the same, the associative algebra structure map ${\assspecies\circ \fieldk\to\fieldk}$ makes $\fieldk$ into a $\coass$-correlation algebra (and similarly for $V$).

Because the unit $\eta$ is an algebra map and the expectation $E$ respects $\eta$, the conditions of Definition~\ref{defi:CProb} are satisfied and we thus have the data of a $\fieldk$-valued homotopy $\coass$-probability space. The cumulant morphism $K$ is made up of the so-called \emph{Boolean cumulants} of the non-commutative (homotopy) probability space. That is, $K_{[n]}$ is the $n$th Boolean cumulant. Thisis essentially the main example of~\cite{DrummondColeParkTerilla:HPTI}.
\item Now assume $V$ is as above but also commutative. Then it is a commutative homotopy probability space in the sense of~\cite{DrummondColeParkTerilla:HPTII}. Again this is supposed to generalize a classical definition. If $V$ is concentrated in degree zero and satisfies two simple inequalities, then it is an algebraic probability space in the sense of~\cite{Tao:APS}. 

As above, the identification of the underlying species of $\comm$ and $\cocomm$ gives maps $\varphi_\fieldk$ and $\varphi_V$ which are defined as in the previous example: ${\commspecies\circ \fieldk\to \fieldk}$ (and likewise for $V$). Altogether then, this is the data of a $\fieldk$-valued homotopy $\cocomm$-probability space. The cumulant morphism $K$ encapsulates the so-called \emph{classical cumulants} of the classical algebraic (or homotopy commutative) probability space. This is essentially the main example of~\cite{DrummondColeParkTerilla:HPTII}.
\end{enumerate}
\end{example}
\begin{remark*}
The cumulants of a probability space (whether classical, Boolean, or free) can be defined combinatorially in terms of M\"obius inversion using an appropriate poset of partitions. One can view the encapsulation of the cumulants of a probability space in terms of operadic algebra as a sort of algebraic enrichment of this combinatorial data, where the choice of cooperad corresponds to the choice of appropriate type of partition. 
\end{remark*}
\subsection{Review of free cumulants}
Independence is a critical feature in probability theory. Cumulants gather the information of a probability space in a way that facilitates the study of independence; the cumulant of a sum of independent random variables is the sum of the individual cumulants. The correct notion for independence in many non-commutative contexts is \emph{free independence}, discovered by Voiculescu~\cite{Voiculescu:SSRFPCSA} (or see the historical survey~\cite{Voiculescu:FPVNAFG}) and studied by many others since then.  We briefly recall free cumulants. See~\cite{NovakSniady:WFC} for a quick overview and~\cite{NicaSpeicher:LCFP} for a more detailed introduction to free cumulants and their connection to free probability theory in general. 
\begin{defi}
A \emph{non-crossing partition} of $N$ is a surjective map $f$ from $[n]$ to $[k]$ such that:
\begin{enumerate}
\item (ordering) if $i<j$ then $\min \left(f^{-1}(i)\right)<\min \left(f^{-1}(j)\right)$ and
\item (non-crossing) $f(1,2,\ldots N)$ is a non-crossing word in $[k]$.
\end{enumerate}
We call $k$ the \emph{size} of $f$.
\end{defi}
\begin{defi}(\cite[11.1]{NicaSpeicher:LCFP})
Let $V$ be a unital $\fieldk$-algebra, let $(\rho_n)_{n\ge 1}$ be a sequence of functionals $V^{\otimes n}\xrightarrow{\rho_n}\fieldk$, and let $f$ be a non-crossing partition of $N$ of size $k$. Then the \emph{multiplicative extension} $\rho_f:V^{\otimes N}\to \fieldk$ is defined as 
\[
\rho_f(a_1\otimes\cdots\otimes a_n)=\prod_{i=1}^k\rho_{|f^{-1}(i)|}(\ul{a_{f^{-1}(i)}}).
\]
Here $\ul{a_{f^{-1}(i)}}$ is the tensor product $a_{j_1}\otimes \cdots\otimes a_{j_{|f^{-1}(i)|}}$ where $j_1,\ldots, j_{|f^{-1}(i)|}$ is the restriction $a_1,\ldots, a_n|f^{-1}(i)$.
\end{defi}
\begin{defi}(\cite[11.4 (3)]{NicaSpeicher:LCFP})
let $(V,E)$ be a non-commutative probability space. The \emph{free cumulants} of $V$ are the unique functions $\{\kappa_N\}$ whose multiplicative extension satisfies the defining equation 
\[
E(a_1\cdots a_N) = \sum_f \kappa_f(a_1\otimes \cdots \otimes a_N)
\]
as $f$ ranges over non-crossing partitions.
\end{defi}
\section{The non-crossing word cooperad and free probability theory}
Finally, we relate non-commutative probability spaces to $\noncrossingcooperad$-correlation algebras and homotopy $\noncrossingcooperad$-probability spaces and show that the cumulant morphism of a $\fieldk$-valued homotopy $\noncrossingcooperad$-probability space recovers the free cumulants defined above.

\begin{defi}
We define a map $\psi$ of species from the underlying species $\Nn$ of the non-crossing words cooperad to the underlying species $\assspecies$ of the associative operad (defined in Example~\ref{example:operaddefinitions}). Under the map $\psi$, a word $w$ in the letters $\{w_1,\ldots, w_{|S|}\}$ goes to the order $f_w$ where $f_w(w_i)=j$ if the subword of $w$ which ends with the first occurence of $w_i$ in $w$ contains $j$ letters from the alphabet. 
\end{defi}
Now, as in the first example above, let $V$ be a unital associative homotopy probability space.

We can give both $\fieldk$ and $V$ the structure of $\noncrossingcooperad$-correlation algebras by composing the map $\psi$ with the structure maps of the associative algebras $V$ and $\fieldk$:
\begin{align*}
 \Nn\circ V\xrightarrow{\psi}\assspecies\circ V\xrightarrow{\text{structure}} V,\\ 
 \Nn\circ \fieldk\xrightarrow{\psi}\assspecies\circ \fieldk\xrightarrow{\text{structure}} \fieldk.
 \end{align*} 
As before, since the map $E$ preserves the unit and the unit is a map of associative algebras, they are compatible with this structure and the whole package is then the data of a $\fieldk$-valued homotopy $\noncrossingcooperad$-probability space.

Now we are ready for the main theorem.
\begin{thm}\label{thm: freever}
Let $(V,E)$ be a non-commutative probability space, viewed as above as a $\fieldk$-valued homotopy $\noncrossingcooperad$-probability space.

Then the cumulant morphism $K$ recovers the free cumulants of the probability space. 
\end{thm}
\begin{proof}
Consider the defining diagram~(\ref{eq: definecumulants}) of the cumulant morphism. By adjunction into vector spaces (or chain complexes), we may restrict the right half of the diagram without losing information, as follows.
\[
\begin{gathered}
\xymatrix{
\noncrossingcooperad\circ \fieldk\ar[r]^{\tilde{\varphi}_{\fieldk}}
&
\noncrossingcooperad\circ\fieldk
\\
\noncrossingcooperad\circ V\ar@{.>}[u]^{\tilde{K}}\ar[r]_{\tilde{\varphi}_V}
&
\noncrossingcooperad\circ V\ar[u]_{\tilde{E}}
}
\end{gathered}
\qquad\Leftrightarrow\qquad
\begin{gathered}
\xymatrix{
\Nn\circ \fieldk\ar[r]^{{\varphi}_{\fieldk}}
&
\fieldk
\\
&V\ar[u]_{E}
\\
\Nn\circ V\ar@{.>}[uu]^{U\tilde{K}}\ar[r]_{U\tilde{\varphi}_V}
&
\Nn\circ V\ar[u]
}
\end{gathered} \qquad\Leftrightarrow\qquad 
\begin{gathered}
\xymatrix{
\Nn \circ \fieldk\ar[r]^{{\varphi}_{\fieldk}}
&
\fieldk
\\
\Nn\circ V\ar@{.>}[u]^{U\tilde{K}}\ar[r]_{\varphi_V}&V\ar[u]_{E}.
}
\end{gathered}
\]
Let $w_N$ be the word $1,\ldots, N$ in the alphabet $[N]$. Define $K_N:V^{\otimes N}\to \fieldk$ in terms of the cumulant morphism as 
\[K_N(z)=K(w_N\otimes z).\]
We will show that the map $K_N$ is precisely the $N$th free cumulant map.

Apply the maps making up the rightmost commutative square to the element of $\Nn\circ V$ represented by $w_N\otimes (v_1\otimes\cdots\otimes v_N)$. The map $\varphi_V$ is just multiplication and so the composition on the bottom and right sides of the square is 
\[
E(v_1\cdots v_N).
\]
Recall the vertical map $\tilde{K}$ is defined as the extension of the cumulant morphism $K: \Nn\circ V\to \fieldk$ as follows:
\[
\xymatrix{
	\Nn\circ V\ar[r]^{U\tilde{K}}\ar[d]_{\Delta_\noncrossingcooperad} &\Nn\circ\fieldk\\
 (\Nn\bcirc \Nn)\circ V\ar[r]_\cong&\Nn\circ (\Nn\circ V)\ar[u]_{\Nn\circ K}.
}
\]
Since $\Nn\circ (\Nn\circ V)$ and $\Nn\circ V$ are defined as colimits (see~Appendix~\ref{appendix:coop}), in order to evaluate the overall composition $\Nn\circ V\xrightarrow{U\tilde{K}}\Nn\circ\fieldk\xrightarrow{\varphi_\fieldk}\fieldk$, it suffices to evaluate on a choice of representatives. That is, let $S$ be the (finite) set of surjections $f$ from $[N]$ to $[M]$ such that $i<j$ implies $\min f^{-1}(i)< \min f^{-1}(j)$ (this set exhausts the isomorphism classes of surjections out of $[N]$). Then the following diagram commutes. The diagram may look intimidating but the right hand side is precisely what we are trying to compute while the left hand side just gives a concrete recipe for the calculation. 
\[
\xymatrix{
\Nn[N]\otimes V^{\otimes [N]}\ar[d]\ar[r]& \Nn\circ V\ar[d]\ar@{=}[r]&\Nn\circ V\ar[ddd]^{U\tilde{K}}
\\
\displaystyle\prod_S \left(\Nn([M])\otimes \bigotimes_{t\in [M]}\Nn(f^{-1}(t))\right)\otimes  V^{\otimes [N]} \ar[d]\ar[r]&(\Nn\bcirc \Nn)\circ V\ar[d]
\\
\displaystyle\bigoplus_S \left(\Nn([M])\otimes \bigotimes_{t\in [M]}\left(\Nn(f^{-1}(t))\otimes V^{\otimes f^{-1}(t)}\right)\right)\ar[r]\ar[d] & \Nn\circ(\Nn\circ V)\ar[d]^{\Nn\circ K}
\\
\displaystyle\bigoplus_S \left(\Nn([M])\otimes \fieldk^{\otimes [M]}\right)\ar[r]\ar[d] & \Nn\circ \fieldk\ar@{=}[r]\ar[d]^\psi& \Nn\circ\fieldk\ar[dd]^{\varphi_{\fieldk}}
\\
\displaystyle\bigoplus_S \left(\assspecies([M])\otimes \fieldk^{\otimes [M]}\right)\ar[d]\ar[r]& \assspecies\circ \fieldk\ar[d]
\\
\fieldk\ar@{=}[r]&\fieldk\ar@{=}[r]&\fieldk 
}
\]
Using the characterization from Corollary~\ref{cor:characterization of nc Delta}, we see that the contribution is $0$ for a function $f$ from $[N]$ to $[M]$ if $f(w_N)$ is crossing. Then the subset of functions from $[N]$ to $[M]$ which contribute to the overall composition coincides precisely with the set of functions from $[N]$ to $[M]$ which are non-crossing partitions.

For a given partition $f$, let us trace the contribution from the $f$ factor in the left side composition. Explicitly, starting with $w_N\otimes V^{\otimes [N]}$, the first vertical map, restricted to the $f$ factor takes this to \[\overline{f(w_N)}\otimes \bigotimes_{t\in [M]} \overline{w_N|f^{-1}(t)}\otimes V^{\otimes [N]}.\]
The second vertical map is just a change of parenthesization on the factor. 

The third vertical map applies $K$ to the factors $\overline{w_N|f^{-1}(t)}\otimes V^{\otimes f^{-1}(t)}$. Because there is no repeated letter in $w_N$, the reduction is trivial, and we can identify $\overline{w_N|f^{-1}(t)}$ with $w_N|f^{-1}(t)$. Then there is an order-preserving isomporphism between $f^{-1}(t)$ and $[|f^{-1}(t)|]$ which realizes $K(\overline{w_N|f^{-1}(t)}\otimes V^{\otimes f^{-1}(t)})$ as 
\[K_N(\ul{
w_N|f^{-1}(t)\otimes V^{\otimes f^{-1}(t)}
}).\]
By construction the map $\psi$ takes $f(w_N)$ to the identity order $[M]\to [M]$ and the final map in the vertical composition is then just the ordered product of the factors corresponding to $f^{-1}(t)$ for $t$ in $[M]$. This product is then
\[\prod_{t=1}^M K_N(\ul{
w_N|f^{-1}(t)\otimes V^{\otimes f^{-1}(t)}
})\]
which is precisely the multiplicative extension of $K_f$ of $(K_1,K_2,\ldots)$.

Thus the overall equation is then 
\[
E(v_1\cdots v_N) = \sum_f K_f(v_1\otimes\cdots\otimes v_N)
\]
which demonstrates that $K_N$ satisfy precisely the same definining equations as the free cumulants $\kappa_N$.
\end{proof}
To conclude the main body, we make two caveats about this approach.
\begin{remark*}\label{remark: problems with this paper}
\begin{enumerate}
\item First of all, this theorem only makes use of the cumulant morphism for very special non-crossing words, those of the form $w_N=1,\ldots, N$. This means that there are many other ``cumulants'' in this context, not only the free cumulants. For example, applying the same methods with the word $w'_N=1,2,\ldots, N-1,N,N-1,\ldots, 3,2$ yields the Boolean cumulants of the same non-commutative probability space. This may be seen either as a feature (flexibility in the method) or a bug (imprecision in the output).
\item More damning is the fact that this method does not seem to work at all in \emph{operator-valued} free probability, where the ground ring is itself non-commutative. In our case, the right hand side of the formula relating expectations and cumulants was a product of individual cumulants $\kappa_n$. But in operator-valued free probability, the right hand side includes nested cumulants, like $\kappa_2(a\kappa_1(b)\otimes c)$. This kind of ``tree-like'' formula does not fit well in this formulism. However, operadic algebra is tailored to describe tree-like compositions and there is a somewhat different and more technical approach using these tools that works in the more general case. This will be studied elsewhere~\cite{DrummondCole:OAOVFC}.
\end{enumerate}
\end{remark*}
\appendix
\section{Unbiased operads and cooperads}\label{appendix:coop}
We will use an unbiased definition for operads and cooperads, as it significantly reduces the notation necessary to describe our structures at the cost of requiring a few explicit definitions rather than a reference. There are several distinct issues that one faces with cooperadic algebra in full generality, related to issues like conilpotency, $0$-ary operations, and the ``handedness'' of the categories we generally work in. We will make several strong simplifying assumptions to avoid the most obvious pitfalls.

Let $\category$ be either the category of vector spaces, the category of graded vector spaces, or the category of chain complexes over $\fieldk$. We consider vector spaces as graded vector spaces concentrated in degree zero and graded vector spaces as chain complexes with zero differential without further comment.

\subsection{Species and plethysm}
\begin{defi}
A linear {\em species} is a functor from finite sets and their isomorphisms to $\category$. A species is \emph{reduced} if it takes value $0$ on the empty set. 

All species will be linear in this paper. 

The \emph{unit species} $\Ii$ has $\Ii(S)=\fieldk$ if $|S|=1$ and $\Ii(S)=0$ otherwise, with the identity for every nonzero morphism.

The \emph{coinvariant composition} or \emph{coinvariant plethysm} of two species $\speciesa$ and $\speciesb$ is the species $\speciesa\circ \speciesb$ given by
\[
(\speciesa\circ \speciesb)(S)=\colim_{S\xrightarrow{f} T} \left(\speciesa(T)\otimes \bigotimes_{t\in T}\speciesb(f^{-1}(t))\right).\]
The \emph{invariant composition} or \emph{invariant plethysm} of two species $\speciesa$ and $\speciesb$ is the species $\speciesa\bcirc \speciesb$ given by
\[
(\speciesa\bcirc \speciesb)(S)=\lim_{S\xrightarrow{f} T}\left(\speciesa(T)\otimes \bigotimes_{t\in T}\speciesb(f^{-1}(t))\right).\]
In both cases the limits and colimits are taken over the diagram category whose objects are maps out of $S$ and whose morphisms are isomorphisms under $S$.
\end{defi}

\begin{lemma}\label{lemma: intertwines the plethysms}
Let $\speciesa$ be a species and let $\speciesb$ be a reduced species. Then there is an isomorphism between $\speciesa\bcirc \speciesb$ and $\speciesa\circ \speciesb$, defined below.
\end{lemma}
\begin{proof}
For a fixed set $S$, choose a set of representatives $\{f_i:S\to T_i\}$, one for each isomorphism type of surjection $f:S\to T$ in the diagram category defining both plethysms. This set is a fortiori finite because we have restricted to surjections.

The invariant plethysm projects onto the defining factor \[(\speciesa\circ \speciesb)(S)_i\coloneqq \speciesa(T_i)\otimes \bigotimes_{t\in T_i}\speciesb(f_i^{-1}(t_i)).\] Likewise, the coinvariant plethysm receives a map from $(\speciesa\circ \speciesb)(S)_i$. 

This collection of maps then determines both: 
\begin{enumerate}
\item a map from the invariant plethysm to the direct product $\prod (\speciesa\circ \speciesb)(S)_i$ and
\item a map from the direct sum $\bigoplus (\speciesa\circ \speciesb)(S)_i$ to the coinvariant plethysm. 
\end{enumerate}
But since the product is finite, the natural map from the sum to the product is invertible and so we can compose to get a map
\[
\speciesa\bcirc \speciesb\to \prod_i (\speciesa\circ \speciesb)(S)_i\cong \bigoplus_i (\speciesa\circ \speciesb)(S)_i\to \speciesa\circ \speciesb.
\]
This overall composition is independent of the choices of representatives. Since $\speciesb$ is reduced, this runs over all isomorphism types necessary to define both  the invariant and coinvariant plethysm. Moreover, because we are working in characteristic zero, the map, for each fixed isomorphism class, is an isomorphism.
\end{proof}

There are two points that require care. First of all, we should make sure that when we actually move between the two, that we consistently adhere to the particular choice of isomorphism outlined here. That is, there are two or three different normalizations of this isomorphism present in the literature. The others differ by something like a factor of $|S|!$ or $\frac{1}{|S!|}$ on each component of the product/sum above). Secondly, we do not have such a map when $\speciesb$ is not reduced.

\begin{lemma}
There are natural isomorphisms making linear species equipped with the unit species and coinvariant plethysm a monoidal category. There are natural isomorphisms making reduced linear species equipped with the unit species and invariant plethysm a monoidal category.
\end{lemma}
\begin{proof}
The left and right unitor isomorphisms can be constructed by direct computation of the c(co)limits involved. 

Colimits (and essentially finite limits) commute with tensor product. Then $(\speciesa\circ \speciesb)\circ \speciesc$ and $\speciesa\circ (\speciesb\circ \speciesc)$ are both naturally isomorphic to 
\[
\colim_{S\stackrel{f}{\twoheadrightarrow} T\stackrel{g}{\twoheadrightarrow} U} \left(\speciesa(U)\otimes \bigotimes_{u\in U} \speciesb(g^{-1}(u))\otimes \bigotimes_{t\in T} \speciesc(f^{-1}(t))\right).
\]
Verifying that these natural isomorphisms satisfy the triangle and pentagon axioms is straightforward. The case of the invariant plethysm is basically the same.
\end{proof}
\subsection{Operads and cooperads}
\begin{defi}
An \emph{operad} is a monoid in the monoidal category of linear species with coinvariant plethysm.
A (reduced) \emph{cooperad} is a comonoid in the monoidal category of reduced species with invariant plethysm.
\end{defi}
The data of an operad $\operad=(\operadspecies,\eta,\mu)$ consists of a species $\operadspecies$ equipped with maps $\eta:\Ii\to\operadspecies$ (the \emph{unit}) and $\operadspecies\circ\operadspecies\xrightarrow{\mu} \operadspecies$ (the \emph{composition}). The composition must be associative and the unit must satisfy left and right unit properties.

More explicitly, to specify a composition map out of the defining colimit of $\operadspecies\circ\operadspecies$ it suffices to give a map out of each term with the appropriate equivariance. So for a map $f:S\to T$, one can specify a map 
\[\mu_{f}:\operadspecies(T)\otimes\bigotimes_{t\in T} \operadspecies(f^{-1}(t))\to \operadspecies(S)\]
and then define the composition map as the colimit of $\mu_f$.

Similarly, the data of a cooperad $\cooperad=(\coopspecies,\epsilon,\Delta)$ consists of a reduced species $\coopspecies$ equipped with maps  $\epsilon:\coopspecies\to I$ (the \emph{counit}) and $\coopspecies\xrightarrow{\Delta}\coopspecies\bcirc\coopspecies$ (the \emph{decomposition}). The decomposition must be coassociative and the counit must satisfy left and right counit properties. 

More explicitly, to specify a decomposition map into the defining limit of $\coopspecies\bcirc\coopspecies$ it suffices to give a map into each term with the appropriate coequivariance. So for a surjection $f:S\twoheadrightarrow T$, one can specify a map 
\[\Delta_{f}:\coopspecies(S)\to \coopspecies(T)\otimes\bigotimes_{t\in T} \coopspecies(f^{-1}(t))\]
and then define $\Delta$ as the limit of $\Delta_{f}$.

In practice, the (co)equivariance and (co)unital conditions are easy to verify and the main thing to check is (co)associativity. 
\begin{remark*}
The expression of operads as monoids in a monoidal category is due to Smirnov~\cite{Smirnov:HTC}; the dual picture was written down in~\cite{GetzlerJones:OHAIIDLS}. In general, biased definitions are more common in the literature. Given a (co)operad in this unbiased definition, one can recover the data of a (co)operad under a more standard definition by restricting to the full subcategory containing only the objects $[n]$.
\end{remark*}
\subsection{Examples}
We shall use a few simple operads and cooperads. In all of the following, 
\begin{enumerate}
\item by definition all the species in the examples are reduced, and sets $S$ are assumed to be non-empty.
\item all units and counits are given by the identity map $\fieldk\to \fieldk$ for each singleton set $S$,
\item it is easy to verify (co)unitality and (co)equivariance, and
\item it is a straightforward (potentially tedious) calculation to verify (co)associativity of the specified (co)composition.
\end{enumerate}
Verifications of (co)unitality, (co)equivariance, and (co)associativity are omitted.
\begin{example}\label{example:operaddefinitions}
\begin{enumerate}
\item The unit species $\Ii$, along with the identity and the canonical isomorphisms $\Ii\bcirc\Ii\cong \Ii\cong \Ii\circ\Ii$, has both an operad and cooperad structure. We denote both of these by $\counitcooperad$.
\item Let $\commspecies$ be the species with $\commspecies(S)=\fieldk$ for all $S$ (and $\commspecies$ applied to all maps is the identity on $\fieldk$). We give this species an operad structure by specifying
\[\mu_f:\fieldk\otimes \bigotimes_{t\in T}\fieldk\to \fieldk\]
given by the natural identification. This is the \emph{commutative operad} and is denoted $\comm$.
\item  Similarly, we give the data $\Delta_f$ for a cooperad with underlying species $\commspecies$. In this case as well,
\[
\Delta_f:\fieldk\to \fieldk\otimes \bigotimes_{t\in T}\fieldk
\]
is the natural identification. This is the \emph{cocommutative cooperad} and is denoted $\cocomm$.
\item Let $\assspecies$ be the species such that $\assspecies(S)$ is the $\fieldk$-linear span of total orders on $S$:
\[\Ord(S)\coloneqq\Iso(S,[|S|]).\] 
We will specify an operad with underlying species $\assspecies$. Given a surjection $f:S\to T$, there is an embedding $\iota_f:\Ord(T)\times \prod \Ord(f^{-1}(t))\to \Ord(S)$ given by
\[
\iota_f\left(\rho\times \prod \tau_t\right)(s) = \tau_{f(s)}(s) + \sum_{\rho(t)<\rho(f(s))} |f^{-1}(t)|
\]
Define the composition map $\mu_f$ as the $\fieldk$-linear extension of $\iota_f$. The resulting operad is the \emph{associative operad}, denoted $\ass$.
\item Finally, we specify a cooperad with the same underlying species $\assspecies$. The decomposition map
\[
\Delta_f:\fieldk\langle \Ord(S)\rangle\to \fieldk\langle \Ord(T)\rangle \otimes \bigotimes_{t\in T}\fieldk\langle \Ord(f^{-1}(t))\rangle.
\]
is again determined by $\iota$ by the equation
\[
\Delta_f(\sigma)=\sum_{\rho,\tau_t} \delta_{\sigma,\iota_f(\rho\times\prod \tau_t)}\left(\rho\times \prod\tau_t\right).
\]
The resulting cooperad is the \emph{coassociative cooperad} and is denoted $\coass$.
\end{enumerate}
\end{example}
\subsection{Algebras and coalgebras}

Now we move on to the discussion of algebras over operads and coalgebras over cooperads. The category $\category$ embeds into the category of (non-reduced) species as follows. Let $V$ be an object in $\category$. Then $\iota(V)$ is the species with $\iota(V)(\emptyset)=V$ and $\iota(V)(S)=0$ for nonempty $S$.

\begin{defi}
Let $\speciesa$ be an species. The \emph{Schur functor} associated to $\speciesa$ is a functor $\category\to\category$, defined by
\[
V \mapsto (\speciesa\circ \iota(V))(\emptyset).
\]
We will abuse notation and use the notation $\speciesa\circ$ for this functor.
\end{defi}

The Schur functor $\Ii\circ$ for the unit species $\Ii$ is naturally equivalent to the identity functor. Since the coinvariant plethysm is associative, the iterated Schur functor of two species is naturally isomorphic to the Schur functor of the plethysm:
\[
 \speciesa\circ (\speciesb\circ (V ))\cong
  (\speciesa\circ \speciesb)\circ(V).
  \] This implies the following.
\begin{lemma}
If the species $\speciesa$ is equipped with an operad structure, the unit and composition induce a monad structure on the Schur functor $\speciesa$.

If the reduced species $\speciesa$ is equipped with a cooperad structure, the counit and cocomposition induce a comonad structure on the Schur functor $\speciesa$.
\end{lemma}

\begin{defi}
Let $\operad=(\operadspecies,\eta,\mu)$ be an operad. An \emph{algebra} over $\operad$ is an algebra over the monad $\operad$. This is the same as a $\category$ object $V$ equipped with a morphism $\operad V\to V$ compatible with the monad structure.

Let $\coop=(\coopspecies,\epsilon,\Delta)$ be a cooperad. A \emph{conilpotent coalgebra} over $\coop$ is a coalgebra over the comonad ${\coopspecies\circ}$. This is the same as a $\category$ object $V$ equipped with a morphism $V\to {\coopspecies}\circ V$ compatible with the comonad structure.
\end{defi}
As is general for monads, the forgetful functor from the category of algebras over an operad $\operad=(\operadspecies,\eta,\mu)$ to $\category$ has a left adjoint, the free $\operad$-algebra functor, realized by the Schur functor and the monad structure of ${\operadspecies}\circ$. We distinguish between the Schur functor $\operadspecies\circ$ between $\category$ and itself and the Schur functor $\operad\circ$ between $\category$ and $\operad$-algebras.

Similarly, the forgetful functor $U$ from conilpotent coalgebras over a cooperad $\coop=(\coopspecies,\epsilon,\Delta)$ to $\category$ has a right adjoint, the cofree conilpotent $\coop$-coalgebra functor, realized by the Schur functor and the comonad structure of ${\coopspecies}\circ$. Again, we distinguish between the Schur functor $\cooperadspecies\circ$ between $\category$ and itself and the Schur functor $\cooperad\circ$ between $\category$ and $\cooperad$-coalgebras.

In any event, the adjunction above implies that a morphism of conilpotent $\coop$-coalgebras from some coalgebra $\coalga$ into $\coop\circ V$ may be identified via this adjoint with a $\category$ morphism from the underlying $\category$-object of $\coalga$ to $V$. 

In general, the adjunction $\Hom_{\category}(U\coalga,V)\to \Hom_{\coop\coalgebras}(\coalga,\coop\circ V)$ is realized by taking a $\category$-morphism $f$ to the composite
\[
\coalga\xrightarrow{\text{coalgebraic structure map}} \coop\circ (U\coalga)\xrightarrow{\coop\circ f} \coop\circ V
\]
and the inverse map is given by taking a coalgebra map to its composite with the counit applied to $V$:
\[
U\coalga\to U(\coop\circ V)\cong \coopspecies\circ V\to \Ii\circ V\cong V.
\]
\subsection{Automorphisms of cofree coalgebras}\label{subsec:auto}

We record a characterization of automorphisms of cofree coalgebras in terms of this adjunction. We call a cooperad \emph{strongly coaugmented} if it is reduced and takes value $\fieldk$ on a singleton. A strongly coaugmented cooperad $\coop$ accepts a map from the cooperad $\counitcooperad$ which fits into the following diagram 
\[
\xymatrix{
	\counitcooperad\ar[rr]^{\id}\ar[dr]&&\counitcooperad\\
	&\coop\ar[ur]_\epsilon
}
\]
which is necessarily unique.

For a species $\coopspecies$, given a $\category$ map $f:\coopspecies\circ V\to V$ and a finite set $S$ we let $f_S$ denote the restriction 
\[\colim_{\Aut S}\coopspecies(S)\otimes V^{\otimes S}\to \coopspecies\circ V\xrightarrow{f} V.\] 

Then we have the following.
\begin{lemma}\label{lemma: identify coalgebra iso}
Let $\coop=(\coopspecies,\epsilon,\Delta)$ be a strongly coaugmented cooperad. Let $V$ be an object of $\category$. A morphism $f:U(\coop\circ V)\cong \coopspecies\circ V\to V$ is adjoint to a coalgebra automorphism $\coop\circ V\to \coop\circ V$ if and only if $f_S$ is an isomorphism when $|S|=1$.
\end{lemma}
\begin{proof}
Let $\tilde{f}$ and $\tilde{g}$ be composable morphisms from $\coop\circ V$ to itself with composite $\tilde{h}=\tilde{g}\circ\tilde{f}$. Write their adjoints from $U(\coop\circ V)$ to $V$ as $f$, $g$, and $h$. Then by using the above characterization of the adjunction, one can calculate that for $S$ a singleton, we have $h_S=g_S\circ f_S$ (identifying $V$ with $\coopspecies(S)\otimes V$). This shows the necessity of the condition.

To show sufficiency, we can proceed by induction on the size of the finite sets in the colimit definining the Schur functor. Let us be a little more explicit for the left inverse to $f$. 

Since we want $\tilde{h}$ to be the identity, we should have $h_S=0$ for $|S|>1$. The explicit formula for $h_S$ contains the term $g_S\circ (f_1^{\otimes S})$ plus a sum of terms each of which involves only $f_{S'}$ and $g_{S''}$ for some $S''$ strictly smaller than $S$. Then by invertibility of $f_1$ this suffices to define $g_S$ recursively. A similar procedure defines a right inverse. A priori the formulas defining the right inverse are different but existence of both one-sided inverses forces them to be equal.
\end{proof}
We conclude the appendix with a few remarks inessential to the flow of the paper.
\begin{remark*}
\begin{enumerate}
\item The proof above explicitly uses the fact that our species are reduced and strongly augmented. In more generality, as long as there is some filtration with good properties (often called weight grading) the same argument works. 
\item Algebras over $\ass$ (respectively $\comm$) are the same thing as associative (associative and commutative) algebra objects in $\category$, justifying the notation.
\item
The reader may have noticed a failure of parallelism, where the coalgebras are conilpotent but the algebras have no dual adjective. This failure of parallelism occurs because we have only used \emph{coinvariant} Schur functor. Even in our restricted setting, the more natural notion for coalgebras over a cooperad would involve an \emph{invariant} Schur functor. As we are interested only in conilpotent coalgebras, the construction here is preferable.
\end{enumerate}
\end{remark*}

\bibliographystyle{amsalpha} 
\bibliography{references-2016}
\end{document}